\documentclass[12pt]{article}
\usepackage{graphicx} 

\usepackage{amssymb}
\usepackage{amsmath,amsthm}
\usepackage[english]{babel}
\usepackage[T1]{fontenc}
\usepackage[a4paper,top=2.5cm,bottom=3cm,left=2.5cm,right=2.5cm,marginparwidth=1.75cm]{geometry}

\usepackage{microtype}

\usepackage{enumitem} %

\usepackage{mathtools}
\usepackage{graphicx,adjustbox}
\usepackage[dvipsnames]{xcolor}
\usepackage{bbm}
\usepackage[colorlinks=true, allcolors=black]{hyperref}
\usepackage[normalem]{ulem}
\usepackage{cancel}
\usepackage{varwidth}  %
\usepackage{booktabs}

\usepackage{xparse}
\NewDocumentCommand{\grad}{e{_^}}{%
  \mathop{}\!%
  \nabla
  \IfValueT{#1}{_{\!#1}}%
  \IfValueT{#2}{^{#2}}%
}

\usepackage{xspace}

\usepackage{dsfont}

\newtheorem{theorem}{Theorem}[section]

\newtheorem{lemma}[theorem]{Lemma}
\newtheorem{proposition}[theorem]{Proposition}

\theoremstyle{definition}
\newtheorem{definition}[theorem]{Definition}
\usepackage{textcomp}
\newtheorem{remarkx}[theorem]{Remark}
\newtheorem{examplex}[theorem]{Example}
\newenvironment{remark}
{\pushQED{\qed}\remarkx} %
{\popQED\endremarkx}
{\popQED\endexamplex}

\DeclareMathAlphabet{\mathup}{OT1}{\familydefault}{m}{n}

\usepackage{ifthen}
\newlength{\leftstackrelawd}
\newlength{\leftstackrelbwd}
\def\leftstackrel#1#2{\settowidth{\leftstackrelawd}%
{${{}^{#1}}$}\settowidth{\leftstackrelbwd}{$#2$}%
\addtolength{\leftstackrelawd}{-\leftstackrelbwd}%
\leavevmode\ifthenelse{\lengthtest{\leftstackrelawd>0pt}}%
{\kern-.5\leftstackrelawd}{}\mathrel{\mathop{#2}\limits^{#1}}}

 \def\calB{{\mathcal B}} 
  \def\calF{{\mathcal F}}

\def\calM{{\mathcal M}}

  \def\calX{{\mathcal X}}

\def\rmd{{\mathrm d}}

 \def\bbE{{\mathbb E}}

 \def\bbN{{\mathbb N}} 
\def\bbP{{\mathbb P}}  \def\bbR{{\mathbb R}}
\def\bbS{{\mathbb S}}  
  
 \def\bbZ{{\mathbb Z}}

\usepackage{subcaption}
\DeclareMathOperator{\Sym}{Sym}
\DeclareMathOperator{\law}{law}
\DeclareMathOperator{\dist}{dist}
\DeclareMathOperator{\vol}{vol}

\title{Random Quadratic Form with random forcing: \\ Metastable synchronization by noise}
\author{Anna Shalova
\\
\normalsize Korteweg-de Vries Institute for Mathematics, University of Amsterdam, \\ \normalsize \href{mailto:a.shalova@uva.nl}{a.shalova@uva.nl}
}
\date{\today}

\begin{document}

\maketitle

\begin{abstract}

    We study the Random Quadratic Form (RQF) on a sphere in the presence of random Brownian forcing. We show that the forcing does not effectively change the law of the process but affects the synchronization properties of the system. While the RQF without forcing exhibits partial synchronization due to the intrinsic symmetries, the introduction of an arbitrarily small forcing results in long-term symmetry breaking and leads to full synchronization. 
    
    In this work we focus on the small forcing regime and recover the multiscale behavior of the two-point process. We show that in the first stage the model converges to an anti-polar configuration due to the symmetries of the RQF and in the second stage the two clusters meet due to the symmetry breaking phenomenon.

    The model is motivated by continuous-time machine learning models such as Neural ODEs and continuous-time formulations of transformers. In particular, the results of this work explain the role of the bias and the scale of its initialization.
\end{abstract}
\tableofcontents
\section{Introduction}
In this work we study the Random Quadratic Form (RQF) on a sphere $\bbS^{n-1}$ with Brownian forcing given by the following stochastic differential equation:
    \begin{equation}
    \label{eq:rqf}
    \rmd X_t  
    = P_{X_t}  \partial Q_t  X_t + \gamma P_{X_t}\partial W_t,
    \end{equation}
    where $P_{X}$ denotes the projection onto the tangent space of $\bbS^{n-1}$ at $X$
    \[
    P_{X} := P_{T_{X}\bbS^{n-1}} = I - XX^T,
    \]
    the noisy process $Q_t: (0, \infty) \times \Omega \to \Sym^n$ is given by 
\begin{equation*}
\label{eq:A}
Q_t =  \frac{1}{\sqrt 2}(B_t + B_t^T), 
\end{equation*}
where $\{B_t^{ij}: i, j \in 1\dots n\}$ are independent Brownian motions and $W_t:(0, \infty) \times \Omega \to \bbR^n$ is an $n$-dimensional Brownian motion independent of $Q_t$. We use  the notation $\partial Q_t, \partial W_t$ to specify that the equation is understood in the Stratonovich sense. 

The RQF model was introduced in \cite{engel2026random} as a stochastic counterpart of a gradient flow of a quadratic functional and only included the multiplicative noise, namely the case $\gamma = 0$ was considered. Analogously to the deterministic setting, the system was shown to exhibit clustering behaviour in the sense that any two solutions of Eq. \eqref{eq:rqf} with $\gamma=0$ driven by the same noisy process in the long-time limit become either aligned or anti-polar. At the same time, the one point motion of the system is a Brownian motion and has no preferred direction, so the nontrivial behavior only appears on the level of the two-point motion. Such a phenomenon is known as \emph{synchronization by noise} and we provide a rigorous formulation of synchronization by noise result for the RQF in Section \ref{sec:intro-main}. 

The anti-polar limiting configuration can be explained by the intrinsic symmetry of the RQF without forcing, which is violated in the presence of an arbitrarily small forcing $\gamma$. As a result of that, the system with any non-zero $\gamma$ in the long-time limit synchronizes to a single point, namely the anti-polar state is no longer stable. 

At the same time, in the small forcing regime $\gamma \downarrow 0$, the strong attraction to the symmetric configuration dominates on the time scale $t\sim \log\gamma^{-1}$ and before converging to the random attractor consisting of a single point, the system approaches the anti-polar configuration defined by the dynamics of the non-forced RQF. We call this phenomenon \emph{metastable synchronization} because both of the limiting configurations, namely stable and metastable ones, are synchronizing. In this work we characterize both of the attractive configurations and the corresponding rates of convergence.

Despite changing the symmetry properties of the system, the given forcing preserves the qualitative behaviour of the single-point process: for an arbitrary $\gamma$, the RQF is a (rescaled) spherical Brownian motion. This, in particular, implies that the effect of the forcing is also only noticeable at the level of the two-point process.

The rest of the paper is structured as follows. In the rest of the introduction we discuss the main driving application from machine learning, give a literature overview and conclude the section with a schematic statement of the main results. In Section \ref{sec:prelim} we give the necessary theoretical background on random dynamical systems and random attractors. We introduce and prove the main results in Section \ref{sec:main}. Finally, in Section \ref{sec:multi} we discuss how the results of this work can be extended to recover the metastable behaviour with multiple ($>2$) scales.
\begin{figure}[t]
    \centering
    \includegraphics[width=0.32\linewidth]{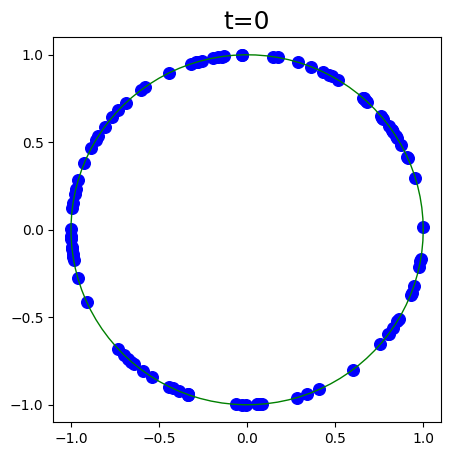}
    \includegraphics[width=0.32\linewidth]{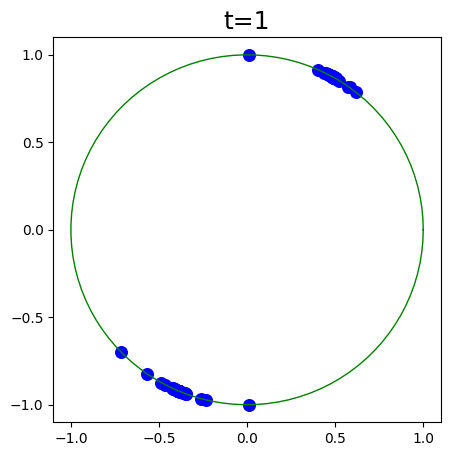}
    \includegraphics[width=0.32\linewidth]{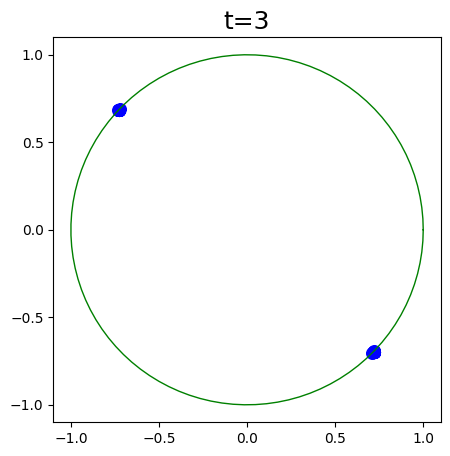}

    \includegraphics[width=0.32\linewidth]{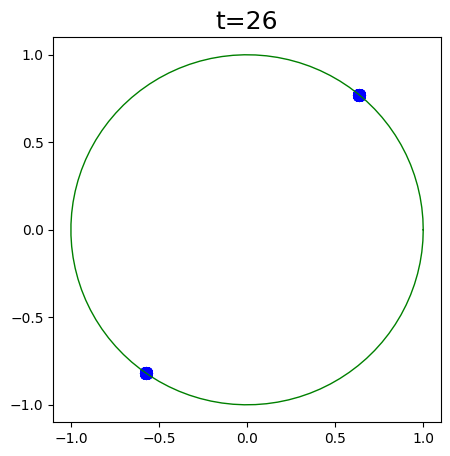}
    \includegraphics[width=0.32\linewidth]{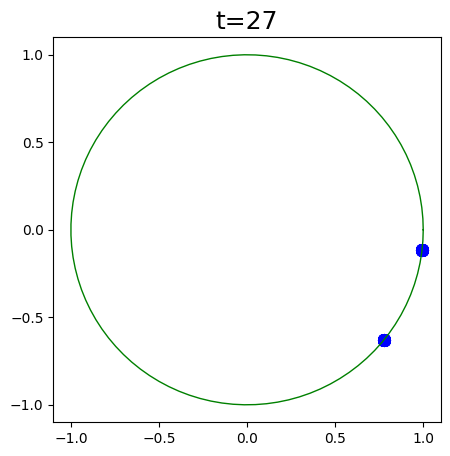}
    \includegraphics[width=0.32\linewidth]{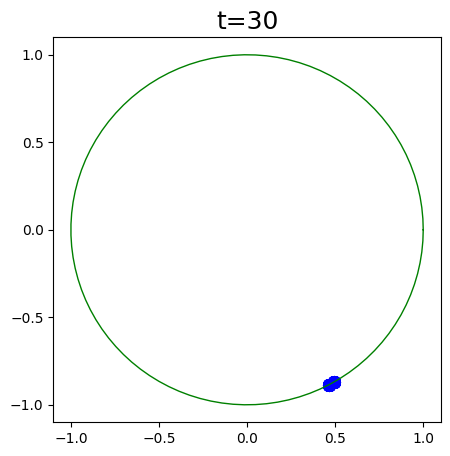}
    \caption{Ensemble of solutions of eq. \eqref{eq:rqf} with different initial conditions driven by the same realization of the noisy processes $Q_t$ and $W_t$ with the forcing parameter $\gamma = 0.1$. In the first stage (the first row), the particles cluster into an anti-polar configuration and follow the dynamics of the 'meta'-attractor $A^Q(\omega)$. In the second stage (the second row), the particles converge to the global random attractor $A(\omega)$ and continue moving as a single cluster.}
    \label{fig:intro}
\end{figure}
\subsection{Neural ODEs and Transformers}
\label{sec:intro-tranformers}
Neural ODEs, introduced in \cite{CRBD18}, are a class of neural networks in which the features $x(t)$ evolve continuously in
time according to an ordinary differential equation
\[
\dot x(t)  = f(\theta(t), x(t), t), \quad x_0\in X_{in},
\]
where $x_0$ is the input data defined on the input space $X_{in}$ and $\theta$ are the parameters of the neural network. This is in contrast to the classical feed-forward networks, in which the features evolution is defined by a discrete-time dynamical system
\begin{equation*}
\label{eq:nn}
x_{k+1}  = f(\theta_k, x_k), \quad x_0\in X_{in}.
\end{equation*}
Considering a specific parametrization of a Neural ODE, consisting of a single feed-forward layer with
a normalization step, the corresponding dynamics of features takes the form
\begin{equation}\label{eq:node}
    \dot x = P_x s\big(M(t)x + B(t)\big), \qquad x(0) = x_0 \in \mathbb{S}^{n-1},
\end{equation}
where $s$ is an activation function and the time-dependent parameters $M(t)$ and $B(t)$
are the weights and the biases of the linear layer. 

Define the cumulative weight and biases processes as 
\[
Q_t =\int_0^t M(s)\rmd s \quad \text{and} \quad W_t =\int_0^t B(s)\rmd s.
\]
Then, considering the linear activation function $s(x) = x$, we note that features in \eqref{eq:node} follow the RQF
dynamics as in Eq. \eqref{eq:rqf} and therefore the RQF can be understood as a simplified model of a Neural ODE model. To justify the white-noise structure of the driving processes we remark that
in discrete-time neural networks, the parameters $\theta_k$ of every layer $k$ are usually initialized randomly and independently from
layer to layer. Hence, we argue that the RQF driven by the diffusion processes $Q_t$ and $W_t$ as defined in \eqref{eq:rqf} is a natural continuous-time proxy for a neural ODE at initialization. The relative scale of the
weight and the bias initialization is then encoded in the single parameter $\gamma$, and,
as we show below, this scale determines the long-time behaviour of the system.

Moreover, recently introduced continuous-time models of transformers \cite{sander2022sinkformers, geshkovski2024mathematical}
can be understood as an extension of the Neural ODE framework and are specifically concerned with the joint dynamics of features corresponding to multiple inputs. In particular, the input of a continuous time transformer is a sequence of vectors $(x_i(0) \in X_{in})_{i= 1\dots N}$, which in language modeling problems correspond to different words in a text. In contrast to Neural ODEs, the dynamics of every feature vector $x^i(t)$ in transformers is additionally coupled to the states of all the other vectors $x^j(t)$ through the so-called Self-Attention mechanism and takes the form
\[
\begin{aligned}
\dot x_i &= P_{x_i}\left(\text{FF}(x_i) + \text{Attn}(x_i; x_1, x_2 \dots x_N)\right), \\
\text{FF}(x_i) &= s(Mx_i + B), \\
\text{Attn}(x_i; x_1, x_2 \dots x_N) &= \frac{1}{\sum_j e^{x_iQ^TKx_j}}\sum_j e^{x_iQ^TKx_j}Vx_j,
\end{aligned}
\]
see \cite{vaswani2017attention, geshkovski2024mathematical} for an extensive description of the architecture. Such a structure can
be interpreted as an interacting particle system, where the Feed-Forward layers act as an effective potential and Self-Attention layers describe the interaction between vectors. 

In the absence of the interaction force, the tokens $(x_i)_{1\le i\le d}$ are driven by the common noise coming from the shared parameters $M$ and  $B$ and thus follow the multipoint motion of the RQF with forcing. Recently, clustering and metastability of tokens in transformers have been extensively studied for the self-attention driven dynamics \cite{geshkovski2024emergence, geshkovski2024mathematical},  see Section~\ref{sec:intro-lit} for details. In this work we follow the approach of \cite{engel2026random} and provide
a counterpart of these results for the dynamics defined purely by the feed-forward layers. In particular, our results imply the existence of metastable clustering of the multi-point motion in Neural ODEs when the variance of the bias initialization is smaller than the variance of the weights initialization.
\subsection{Literature overview}
\label{sec:intro-lit}
\paragraph{Diffusions on a sphere.}
Diffusion processes on $\bbS^{n-1}$ naturally arise in the context of stability analysis of stochastic differential equations by means of the multiplicative ergodic theorem. In particular, the  Furstenberg–Khasminskii formula reduces the calculation of the leading Lyapunov exponent to the ergodic averaging of a specific functional over the ergodic measure of the projective process \cite{furstenberg1963noncommuting, khasminskii2011stochastic}, see also \cite[Chapter 6]{Arnold98} and \cite{imkeller2001some}. Therefore, synchronizing behaviour of diffusions on the sphere is closely related to the Lyapunov stability of SDEs in Euclidean spaces and has been studied in various formulations. Specifically, synchronization of diffusions on $\bbS^{n-1}$ has been established: for the canonical Brownian motion in \cite{baxendale1986asymptotic}, for general isotropic Brownian flows in   \cite{raimond1999flots, cranston2016weak} and for the RQF formulation in \cite{engel2026random}. We remark that for each of these models the one-point motion is a spherical
Brownian motion and therefore carries no distinctive information about the flow, which motivates the study of the two-point process.

We remark that the Euclidean counterpart of this picture is classical: the full Lyapunov spectrum for isotropic Brownian flows on $\mathbb{R}^d$ is computed in \cite{LeJ85}, see also \cite{BH86}. And the random attractor in the stable regime is shown to be a singleton \cite{LeJ85,DLJ88}. 

\paragraph{Synchronization by noise.}
Synchronization of diffusion processes is an example of a more general phenomenon which is known as synchronization by noise. For a general class of random dynamical systems it has been formalized in \cite{flandoli2017synchronization}, and we refer the reader to this work for an extensive literature review on the topic. The approaches allowing to establish synchronization by noise include multiplicative ergodic theorem \cite{arnold1983stabilization, baxendale1991statistical, caraballo2004stabilisation} and the Feller explosion test \cite{scheutzow2002comparison, cranston2016weak} which was also used for the RQF without forcing in \cite{engel2026random}. Alternatively, synchronization can be deduced from the order-preserving properties of the dynamics \cite{CrauelFlandoli98, FGS2}.
Local synchronization can also be established using large-deviations theory, see e.g. \cite{mahony1996gradient, tearne2008collapse}.

\paragraph{Metastability in particle systems.} Classically, metastability of a stochastic system refers to the phenomena observed in gradient-type systems when the system spends a large time near the local minima of the corresponding energy. In such a setting the metastability can be studied using a large deviations-based approach and, in particular, Eyring-Kramers asymptotics \cite{FW12,BG06,BdH15}. This work is concerned with a different setting, namely small perturbation of a \emph{stochastic} process but not of a deterministic flow. At the same time, our result can be understood as an analog of the classical metastability phenomenon but on the level of the \emph{two-point motion}.  

Since we characterize the relative convergence of multi-point motion, another related phenomenon is the transient clustering in interacting
particle systems where the fast scale corresponds to the clusters formation. In particular, on a finite-particle level, the metastable clustering of interacting diffusions and the corresponding Lyapunov exponents are studied in \cite{adams2026separation}. In a mean-field setting, the cluster formation and merging dynamics is discussed in \cite{gerber2026formation, wehlitz2026energetic}. We also remark that the mean-field formulation is a special case of the dynamic metastability framework \cite{otto2007slow}. However, the clustering mechanism studied in these works is of a different nature and is caused by the pairwise interactions. Note that interactions are not present in our setting.

\paragraph{Clustering and metastability in transformers.}
The continuous-time models of transformers can be interpreted as interacting particle systems \cite{sander2022sinkformers, geshkovski2024emergence, geshkovski2024mathematical} and have been shown to exhibit (metastable) clustering dynamics. In particular, the long-time clustering of tokens has been established in various settings, see \cite{rigollet2025mean} for an overview
of the results and recent works \cite{burger2025analysis, alcalde2026quantifying, li2026diverse}. In particular, the rate of convergence to a single cluster in this context is established in \cite{chen2025quantitative}. The transient clustering dynamics in the context of transformers is studied in \cite{geshkovski2024dynamic, bruno2024emergence, bruno2025multiscale}. 

Closest to the present work are the
stochastic formulations of Neural ODEs and transformers, in which the randomness appears due to the random initialization of the parameters. In particular, the pathwise synchronization of the flow of randomly initialized Neural ODEs is characterized in \cite{engel2026random, agazzi2026stochastic}. Similarly, clustering in random attention-based models is established in \cite{fedorov2026clustering, koubbi2026homogenized}. We also remark that the pathwise synchronization is conceptually different from the clustering in noisy transformer models \cite{shalova2026solutions, balasubramanian2025structure}, which correspond to interacting particle systems with independent noises.

\subsection{Main Results}
\label{sec:intro-main}
In this work we study the effect of additive noise on the two-point process of the RQF and the metastable behaviour arising in the small-forcing regime $\gamma \to 0$. We characterize both the global attractor and the transient clustering dynamics appearing on the faster time-scale. We complement the analysis by studying the one-point motion and showing that no metastability arises on the level of one-point dynamics.

To introduce the results, we will require some preliminary facts. Recall that the processes $Q_t$ and $W_t$ are independent and therefore the probability space generated by $(Q_t, W_t)$ is in fact a product space $\Omega = \Omega^Q \times \Omega^W$, where $Q_t  = Q_t(\omega^Q)$ and $W_t  = W_t(\omega^W)$. In case $\gamma = 0$, the space $\Omega^W$ is redundant and easily factors out. In addition, the SDE \eqref{eq:rqf} has smooth coefficients and thus admits a path-wise solution, see Proposition \ref{prop:SDE}. We denote the corresponding solution map by $X_t^\omega(x_0)$, where $x_0 \in \bbS^{n-1}$ denotes the initial condition $X_0 = x_0$. 

With this notation we are ready to present the main results. In particular, recall the characterization of the two-point process without forcing from \cite{engel2026random}.
\begin{theorem}[$\gamma =0$, {\cite[Theorem 4.8]{engel2026random}}]
\label{th:intro-0}
There exists a random set $A^Q(\omega) = A^Q(\omega^Q)$ consisting of two anti-polar points
\[
A^Q(\omega) = \{a^Q(\omega), -a^Q(\omega)\},
\]
where $a^Q(\omega)$ is measurable with respect to the past, and for all $x_0 \in \bbS^{n-1}$ the pathwise solution $X_t^{\omega} (x_0)$ in the long-time limit converges to $A^Q(\theta_t\omega)$:
\[
\dist (X_t^{\omega} (x_0), A^Q(\theta_t\omega)) \to 0 \quad \Omega-\text{a.s.},
\]
where $\theta_t$ is the time-shift defined in Eq. \eqref{eq:shift}.
\end{theorem}
The random set $A^Q$ is in fact the forward attractor of the corresponding random dynamical system as defined in Section \ref{sec:attractor}. In case of the non-zero forcing, $A^Q(\omega)$ no longer attracts all the trajectories in the long-time limit due to the effect of the additive noise. Instead, all the trajectories synchronize to a single-valued attractor $A(\omega)$. However, on a time-scale $t\sim \log |\gamma|^{-1}$, the set $A^Q(\omega)$ is still attractive, and thus in the forced regime we call $A^Q(\omega)$ the 'meta'-attractor after the 'meta'-stable behaviour it describes. We, therefore, obtain the following characterization of the dynamics for $\gamma \neq 0$.
\begin{theorem}[$\gamma \neq 0$]
\label{th:intro-gamma}
There exists a random set $A^Q(\omega^Q)$ as in Theorem \ref{th:intro-0} and a random singleton $ A(\omega) = \{a(\omega)\}$ such that the dynamics consists of the two stages:
\begin{itemize}
    \item (convergence to the 'meta'-attractor $A^Q$):
    \[\bbE \dist (X_t^{\omega} (x_0), A^Q(\theta_t\omega)) \lesssim e^{-\frac{1}{2}t} + |\gamma| C(n, t), \]
    \item (convergence to the global attractor $A$):
    \[\bbE \dist (X_t^{\omega} (x_0), A(\theta_t\omega)) \lesssim e^{-\lambda(n, \gamma)t}, \quad \lambda(n, \gamma) \sim \frac{1}{2}\gamma^2(n-1),
    \]
\end{itemize} 
where $\theta_t$ is the time-shift as in Eq. \eqref{eq:shift} and $C(n, t) \lesssim \sqrt{n}e^{2t}$ for large $t$. The characterization holds for an arbitrary choice of $x_0$.
\end{theorem}
In particular we obtain explicit convergence rates in expectation to the random attractor in both forced and non-forced regimes, which allows us to study the intermediate regime of convergence to the anti-polar configuration.
We highlight that the two attractors $A^Q$ and $A$ are structurally different and therefore the long-term behaviour of the two-point motion of the system with forcing differs from the non-forced case and exhibits a multi-scale behavior, see Figure \ref{fig:intro}. At the same time, this difference cannot be detected on the level of one point motion as follows from the following result, see also Theorem \ref{th:main-distributional}.
\begin{theorem}[RQF is a Brownian motion] For any $\gamma \in \bbR$ the process $X_{t/(1+\gamma^2)}$ is an $\bbS^{n-1}$-valued Brownian motion.
\end{theorem}
\paragraph{Acknowledgments and use of AI.} I am grateful to Maximilian Engel and Enrique Carro Garrido for many insightful discussions on random attractors. The work was supported by the Dutch Research Council (NWO), in the framework of the program VI.Vidi.233.133 `A Rigorous Framework for Transient Random Dynamics'. 

I used AI at different stages of the manuscript preparation. In particular, the use of a regularized function in Lemma \ref{lem:lyap-gamma} and Proposition \ref{lem:gamma-0} was suggested by an AI and the code used to produce illustrations was generated by an LLM. The paper was written entirely by me and all mistakes are mine. 
\section{Notation and Preliminaries}
\label{sec:prelim}
 In this section we introduce the notation and give the necessary theoretical background on SDEs and random dynamical systems. We refer the reader to \cite{Arnold98} for a detailed introduction to RDS. 

For a separable metric space $\calX$ we denote the Borel $\sigma$-algebra on $\calX$ by $\calB(\calX)$. Let $\Omega^Q = C_0(\bbR, \bbR^{n\times n})$ and $\Omega^W = C_0(\bbR, \bbR^{n})$ be the spaces of continuous functions satisfying $\omega(0)=0$ equipped with the metric $d$:
\[
d(\omega, \widehat{\omega}):=\sum_{N=1}^{\infty} \frac{1}{2^N} \frac{\|\omega-\widehat{\omega}\|_N}{1+\|\omega-\widehat{\omega}\|_N}, \quad \|\omega-\widehat{\omega}\|_N:=\sup _{|t| \leq N}\|\omega(t)-\widehat{\omega}(t)\|,
\]
and let $\mathcal{F}^Q =\calB(\Omega^Q), \ \calF^W =\calB(\Omega^W)$. Let $\bbP^Q$ and $\bbP^W$ be the Wiener probability measures on $\mathcal{F}^Q$ and $\calF^W$, where the Wiener probability measure of an $\bbR^m$-valued Brownian motion is given by
\[
\mathbb{P}\left(\left\{\omega \in \Omega: \omega_1(t) \leq x_1, \ldots, \omega_m(t) \leq x_m\right\}\right)=\frac{1}{(2 \pi |t|)^{m / 2}} \int_{-\infty}^{x_1} \cdots \int_{-\infty}^{x_m} e^{-\|y\|^2 / 2|t|} \mathrm{d} y_1 \cdots \mathrm{~d} y_m,
\]
for all $x \in \mathbb{R}^m$.
We denote the product probability space by $(\Omega, \calF, \bbP) = (\Omega^Q \times \Omega^W, \calF^Q \times \calF^W, \bbP^Q\times\bbP^W)$. 

Finally, we define the family of time shifts $\left(\theta_t\right)_{t \in \bbR}$ on the product space $(\Omega, \mathcal{F})$ by
\begin{equation}
\label{eq:shift}
    \theta_t \omega(\cdot):=\omega(t+\cdot)-\omega(t),
\end{equation}
and remark that this family preserves the Wiener measure of a set $A\in \calF$, namely $\bbP(\theta_t^{-1} A) =\bbP(A) $ for all $t\in \bbR, \ A \in \calF$.

\subsection{Preliminaries on SDEs.}
Let $\calM$ be a smooth Riemannian manifold without boundary. Consider an $\calM$-valued Ito SDE of the form
\begin{equation}
\label{eq:sde}
\rmd X_t=F_0\left(X_t\right) \rmd t+\sum_{j=1}^m F_j\left(X_t\right) \rmd W_t^j, \quad X_0=x \in \calM, t \in \mathbb{R}.
\end{equation}
Fixing a probability space of the Brownian motion $W_t$ we define the pathwise solution of \eqref{eq:sde} as follows.
\begin{definition}[Pathwise solution]
    Given a probability space generated by an $m$-dimensional two-sided Brownian motion $(\Omega, \calF, \bbP)$, an SDE \eqref{eq:sde} is said to admit a pathwise solution on $(\Omega, \calF, \bbP)$ with initial condition $x_0 \in \calM$ if there exists a map $X^\omega: \Omega \to C(\bbR, \calM)$ satisfying $\omega$-a.s
    \[
  X^\omega_t= x_0 + \int_0^t F_0\left(X^\omega_s(x_0)\right) \rmd s +\sum_{j=1}^m \int_0^t F_j\left(X^\omega_s(x_0)\right) \rmd W^{\omega, j}_s, \quad X_0^\omega = x_0,
    \]
   where $ W^{\omega,j}_t$ is the $j$-th component of the sample path $W^\omega_t$.
\end{definition}
\begin{definition}[Infinitesimal generator]
    Let the diffusion $X_t$ be a weak solution of \eqref{eq:sde}, the operator $L: C^\infty(\calM) \to C^\infty(\calM)$ 
\[
(Lf)(x) := \lim_{t\downarrow 0}\frac{1}{t} \bbE \left[f(X_t) - f(X_0)\big| X_0 = x\right]
\]
 is called the \emph{infinitesimal generator} of $X_t$. 
\end{definition}

The infinitesimal generator describes the evolution of statistics of a diffusion process, and thus can be used to track the expected distance between two solutions along the dynamics. This is the key component of the proof of Theorems \ref{th:main} and \ref{th:main-meta},  where we will require the following elementary version of Dynkin's formula.
\begin{proposition}[Dynkin's formula, {\cite[Theorem 7.4.1]{oksendal2003stochastic}}]
\label{prop:dynkin}
Let $L$ be the infinitesimal generator of a diffusion process $X_t$ on $\bbR$ with $X_0 = x_0$, then for any $f\in C_c^\infty$ and for all $t\in \bbR_+$
\[
\bbE f(X_t) = f(x_0) + \bbE\left(\int_0^t (Lf)(X_s)\rmd s\right).
\]  
\end{proposition}

Finally, we remark that infinitesimal generator $L$ and its $L^2$ adjoint $L^*$ define the backward and forward Kolmogorov evolutions respectively. The latter, 
\[
\partial_t \rho_t - L^*\rho_t =0,
\] 
describes the evolution of $\rho_t= \law X_t$ and is also known as the Fokker-Planck equation.

\subsection{Random Dynamical Systems}
\label{sec:rds}
\begin{definition}[Random dynamical system (RDS)] Let $(\Omega, \calF, \bbP)$ be an abstract probability space and $(\calM, \calB(\calM))$ be a compact Riemannian manifold with the corresponding Borel $\sigma$-algebra. An RDS consists of the two components:
\begin{enumerate}
    \item \emph{model of the noise:} a family of $\calF$-measurable measure-preserving maps $(\theta_t: \Omega \to \Omega)_{t \in \mathbb R}$, satisfying:
    \[
    \begin{aligned}
        \theta_0\omega &= \omega, \ \forall \omega\in \Omega, \\  \theta_{t+s}\omega &= \theta_{t}\theta_{s}\omega, \ \forall t, s \in \bbR, \omega \in \Omega,
        \end{aligned}
    \]

\item \emph{model of the dynamics:} a $(\calB(\bbR) \otimes \calF \otimes \calB(\calM))$-measurable map $\varphi: \bbR \times \Omega \times \calM \to \calM$ which, for all $\omega \in \Omega$ and $x \in \calM$ satisfies the \emph{cocycle property} 
\begin{equation*} \label{eq:cocycle}
    \varphi(0, \omega, x) = x, \ \text{ and } \ \varphi(t+s, \omega, x) = \varphi(t, \theta_s\omega, \varphi(s, \omega, x)), \ \forall s, t \in \bbR.
\end{equation*} 
\end{enumerate}
\end{definition}
Given an RDS $(\varphi, \theta)$, for any $u,v \in \mathbb R$ with $u < v$, we denote by $\mathcal{F}_{u,v}$ the sub-$\sigma$-algebra generated by  the random variables $\varphi(t, \theta_s \omega, x)$ for $x \in \calM$ and $t, s \in \mathbb R$ with $u\leq s\leq v$ and $0< t \leq v-s$. We say that the RDS satisfies the Markov property if its future is independent of its past, namely:
\begin{definition}[Markov RDS]
The RDS $(\theta, \varphi)$ is called \emph{Markov} if $\mathcal F_{-\infty,0}$ and $\mathcal{F}_{0,\infty}$ are independent.
\end{definition}
Due to the following result, any SDE with sufficiently regular coefficients can be uniquely reformulated in the form of a Markov RDS.
\begin{proposition}[SDE as an RDS, {\cite[Theorem 2.3.42]{Arnold98}}] \label{prop:SDE}
Consider an Ito SDE \eqref{eq:sde} or the corresponding Stratonovich SDE and assume that $F_0 \in C^{1, \delta}(\calM, T\calM)$ and every $F_j \in C^{2, \delta}(\calM, T\calM)$  for some $\delta \in(0,1]$. Let $(\Omega, \calF, \bbP)$ be a (fixed) probability space generated by a two-sided $m$-dimensional Brownian motion.
Then there exists a unique $(\calB(\bbR) \otimes \calF \otimes \calB(\calM))$-measurable map $\varphi(t, \omega, x)$ such that
\begin{itemize}
\item $X_t^\omega =\varphi(t, \omega, x)$ is a pathwise solution of \eqref{eq:sde},
    \item $(\theta, \varphi)$ is a Markov RDS, where $(\theta_t)_{t\in \bbR}$ is the family of time shifts defined in Eq. \eqref{eq:shift}.
\end{itemize}
\end{proposition}
Due to their equivalence, we switch to the RDS notation $\varphi(t, \omega, x)$ for the pathwise solutions of Eq. \eqref{eq:sde}.

Moreover, to use the probabilistic estimates of Dynkin's type for Markov RDS with random initial conditions we will need the following property of conditional expectation.

\begin{lemma}[Freezing lemma]
\label{lem:freezing}
Let $(\theta, \hat \varphi)$ be a Markov RDS on a manifold $\calM$ over the probability space $(\Omega, \calF, \bbP)$ and let $\beta: \Omega \to \calM$ be an $\calF_{-\infty, 0}$-measurable random point. Let $\psi: C(\bbR_+, \calM) \to \bbR$ be a Borel-measurable bounded map and define 
\[
\Psi(x) :=\bbE[\psi(\hat\varphi(\cdot, \omega, x))], \quad \forall x\in \calM.
\]
Then $\Psi$ is Borel-measurable, bounded and
\[
\bbE[\psi(\hat\varphi(\cdot, \omega, \beta(\omega)))\big| \calF_{-\infty, 0}] = \Psi(\beta(\omega)), \quad \Omega-\text{a.s.}.
\]
In particular, $\Psi(\beta(\omega))$ is $\calF_{-\infty, 0}$-measurable and satisfies
\[
\inf_{x}\Psi(x) \leq \bbE[\Psi(\beta(\omega))] \leq \sup_{x}\Psi(x).
\]
\end{lemma}
\begin{proof}
    Since $\hat \varphi$ is Markov, the map $(\omega, x) \mapsto \hat \varphi(\cdot, \omega, x)$ is $\calF_{0, \infty} \times \calB(\calM)$-measurable, implying that the same is true for the composition function $\psi(\hat \varphi(\cdot, \omega, x))$. With this remark, the result is a special case of the 'freezing' lemma \cite[Lemma 4.1]{baldi2017stochastic}.
\end{proof}
In particular, the freezing lemma decouples expectations with respect to the past and future when the corresponding $\sigma$-algebras are independent.
\subsection{Random Attractors}
\label{sec:attractor}
Given a state space $\calM$ and a probability space $(\Omega, \calF, \bbP)$, we say that a set-valued map $A: \Omega \to \calB(\calM)$ is a \emph{random compact set} if it is $\Omega$-a.s. compact and the function $\dist(x_0, A(\omega))$ is $\calF$-measurable for all $x_0 \in\calM$. With this, we define a random point attractor of an RDS as the random set attracting all trajectories under the forward dynamics, namely
\begin{definition}[Random Point Attractor]\label{def:Attr}
   A random compact set $(A(\omega))_{\omega \in \Omega}$ is called the \emph{forward point attractor}, if 
   \begin{itemize}
       \item it is $\varphi$-invariant:
       \[
       \varphi(t, \omega, A(\omega)) = A(\theta_t\omega), \quad \forall t\in\bbR,
       \]
   \item for every $x \in \calM$
            $$\dist\big(\varphi(t, \omega, x), A(\theta_t \omega)\big) \to 0, \quad \Omega-\text{a.s.}$$
            \end{itemize}
    \end{definition}
    A random attractor $A(\omega)$ is called \emph{minimal} if for any other random attractor $\tilde A(\omega)$ we have $A(\omega) \subseteq \tilde A(\omega)$, $\Omega$-a.s. Finally, replacing the a.s.
 convergence with convergence in probability we obtain the definition of a \emph{weak random point attractor}. We note that if a strong point attractor exists, it is also a weak point attractor but the converse is not true. Existence of a weak point attractor for ergodic Markov RDS on compact state spaces is guaranteed by the correspondence theorem, see \cite[Theorem 4.2.9]{KuksinShirikyan12}. In particular, the following holds.
 
 \begin{proposition}[Existence of a weak point attractor]
 \label{prop:ExWeakPntAttr}
 Let $(\theta, \varphi)$ be a Markov RDS on a compact Riemannian manifold which admits a unique ergodic measure $\rho$, then  the weak limit 
$$ \varphi(t, \theta_{-t}\omega, \cdot)^* \rho \stackrel{w}{\to} \mu_\omega,  \ \text{ as } t \to \infty,$$
exists almost surely and the random set
 $$A(\omega) := \operatorname{supp}(\mu_\omega)$$
 is the minimal weak point attractor of the RDS $(\theta, \varphi)$.
\end{proposition}
As follows from the result by Le Jan \cite{LeJan1987}, the sample measures can either consist entirely of atoms or be fully continuous with no atoms at all, namely the following characterization holds.
\begin{proposition}[Discrete vs. continuous sample measures {\cite{LeJan1987}}]
\label{prop:discrete}
    The sample measures $\mu_\omega$ are either $\Omega$-a.s. continuous, namely for every $x\in\calM$ satisfy $\mu_\omega(\{x\}) = 0$, or are $\Omega$-a.s. discrete measures supported on $N\in\bbN$ atoms and given by
\begin{equation*}
\mu_{\omega} = \frac{1}{N} \sum_{i=1}^N \delta_{a_i(\omega)},
\end{equation*} 
where each $a_i:\Omega \to \calM$ is an $\mathcal F_{- \infty, 0}$-measurable random point.
\end{proposition} 

\section{Main results}
\label{sec:main}
We begin by recovering the Fokker-Planck equation of the RQF SDE with forcing. We establish that RQF is a rescaled spherical Brownian motion, which, in particular, implies that it is ergodic and the unique invariant distribution is the uniform measure on $\bbS^{n-1}$.

\begin{theorem}[RQF is a Brownian motion]
\label{th:main-distributional}
Let $\rho_t = \law (X_t)$ as in the eq. \eqref{eq:rqf}, then for any $\gamma \in \bbR$, $\rho_t$ is the unique classical solution of the rescaled heat flow on $\bbS^{n-1}$:
\[
\partial_t \rho_t - \frac{1 + \gamma^2}{2} \Delta \rho_t = 0, \quad \rho_t \stackrel{w}{\to} \law (X_0) \text{ as } t \downarrow 0,
\]
where $\Delta$ is the Laplace-Beltrami operator on $\bbS^{n-1}$. In particular, the uniform measure $\bar\rho = \frac{1}{\vol_{\bbS^{n-1}} (\bbS^{n-1})}\vol_{\bbS^{n-1}}$ is the unique ergodic measure of the RQF, where $\vol_{\bbS^{n-1}}$ is the volume measure of the $n-1$-dimensional sphere.
\end{theorem}
\begin{proof}
    Consider the following Stratonovich diffusions on $\bbS^{n-1}$:
    \[
    \begin{aligned}
        dU_t &=  P_{U_t} \partial Q_t U_t, \\
        dV_t &= P_{V_t} \partial W_t.
    \end{aligned}
    \]
    The process $V_t$ is a natural definition of a Brownian motion on $\bbS^{n-1}$. Moreover, the generator of the process $U_t$ is $L_U = \frac{1}{2}\Delta$, see \cite[Theorem 4.3]{engel2026random}. Since the driving processes $Q_t$ and $W_t$ are independent, we conclude that the generator of the RQF with forcing takes the form
    \[
    Lf = (L_U + \gamma^2 L_V)f = \frac{(1+\gamma^2)}{2}\Delta f, \quad \forall f \in C^\infty.
    \]
    Finally, since the Laplace-Beltrami operator on $\bbS^{n-1}$ is essentially self-adjoint we conclude that $\law X_t$ solves the given Fokker-Planck equation and hence the result.
\end{proof}

\subsection{Random attractor}
We now move to studying the $\omega$-pointwise properties of the solutions to \eqref{eq:rqf}. For this we define an essential component of the proof, namely the auxiliary process $Z_t = \left<X_t, Y_t\right>$, where the couple $(X_t, Y_t)$ solving:
\begin{equation}
\begin{split}
\label{eq:rqf-two}
    \rmd X_t  
    &= P_{X_t}  \partial Q_t  X_t + \gamma P_{X_t} \partial W_t, \quad X_0 \in \bbS^{n-1} , \\
    \rmd Y_t 
    &= P_{Y_t}  \partial Q_t  Y_t + \gamma P_{Y_t} \partial W_t, \quad \ \  Y_0 \in \bbS^{n-1}
\end{split}
\end{equation}
is the two-point motion of the RQF with forcing. 
Note that any two points on a sphere $x, y \in \bbS^{n-1}$ coincide iff $\left<x, y\right> =1$ and are anti-aligned iff $\left<x, y\right> =-1$. 
Therefore, convergence to either an anti-polar configuration or a singleton of a stochastic process can be both characterized in terms of the dynamics of the corresponding process $Z_t$, namely its boundary behaviour at $\pm 1$. In particular, for $\gamma \neq 0$ we will require the following moment-bound on $(1-Z_t)$ to upper bound the distance to the attractor.
\begin{lemma}[Lyapunov function for $\gamma\neq 0$]
\label{lem:lyap-gamma}
    Let $(X_t, Y_t)$ be the two-point motion of the RQF with forcing $\gamma\neq 0$ and let $Z_t = \left<X_t, Y_t\right>$. Then for all $p \in(0, \frac{\gamma^2(n-1)+4}{2(\gamma^2 + 4)})$ and every $z_0 = \left<X_0, Y_0\right> \in [-1, 1]$:
    \[
    \bbE (1-Z_t)^p \leq  (1-z_0)^p e^{-p\alpha t},
    \]
    where 
    \[
    \alpha = \min\{\gamma^2(n-1), \gamma^2(n-1)+ 4 - 2p(\gamma^2 + 4))\},
    \]
in particular, $\alpha >0$.
\end{lemma}
The proof of the bound is largely technical and therefore is deferred to Section \ref{sec:proofs}. 
In addition, we will need the following pointwise bound.
\begin{lemma} 
\label{lem:comparison}
For $x, y \in \bbS^{n-1}$ let $z= \left<x, y\right>$, then $\forall p\in(0, \frac{1}{2}]$:
\[
\dist(x, y) \leq \frac{\pi}{2^p}(1-z)^p.
\]
\end{lemma}
\begin{proof}
First note that the geodesic distance on a sphere is upper bounded by the Euclidean distance $\dist(x, y)\leq \frac{\pi}{2}\|x-y\|$. At the same time, expanding the norm of the difference we get
\[
\|x-y\|^2 = \|x\|^2 + \|y\|^2 - 2\left<x, y\right> = 2(1-z),
\]
and since $x, y\in \bbS^{n-1}$ implies $1 - z\in [0,2]$:
\[
\|x - y\| = \sqrt{2}(1-z)^{\frac{1}{2}}  = \sqrt{2}(1-z)^{\frac{1}{2} -p}(1-z)^{p}\leq 2^{1 - p} (1-z)^{p}.
\]  
Combining the inequalities we get the result.
\end{proof}
We now proceed to characterizing attractors of the RQF with forcing.
\begin{theorem}[Random attractor with forcing]
\label{th:main}
Let $\gamma \neq 0$, then there exists an $\calF_{-\infty, 0}$-measurable map $a:\Omega \to \bbS^{n-1}$ such that $A(\omega) = \{a(\omega)\}$ is the strong forward point attractor of the RQF RDS. Moreover, for all $x_0\in\bbS^{n-1}$:
\begin{equation}
\label{eq:main-estimate}
\bbE(\dist(\varphi(t, \omega, x_0), A(\theta_t\omega))) \leq \pi e^{-\lambda_\gamma t},
\end{equation}
where $\lambda_\gamma = \frac{2\gamma^2}{\gamma^2 +4}(n-1)$.
\end{theorem}
\begin{proof}
The proof is structured as follows. First we prove almost sure collapse $\dist (X_t, Y_t) \to 0$ using the Lemma \ref{lem:lyap-gamma}. Then we relate the two-point dynamics to the weak forward random point attractor of the corresponding RDS. Finally, we apply the freezing Lemma \ref{lem:freezing} to establish almost sure convergence and the estimate \eqref{eq:main-estimate}.

\emph{Step 1: Lyapunov function.} 
Applying Lemma \ref{lem:lyap-gamma} with $p = \frac{2}{\gamma^2 +4}$, we obtain $\alpha(\gamma, n, p)= \gamma^2(n-1)$ and therefore:
\[
\bbE (1-Z_t)^p \leq  (1-z_0)^p e^{-\lambda_\gamma t},
\]
where $\lambda_\gamma = \frac{2\gamma^2}{\gamma^2 +4}(n-1)$. Since the bound holds for an arbitrary choice of the initial condition $z_0\in[-1, 1]$, the stochastic process $\xi_t = (1-Z_t)^p$ is a non-negative supermartingale. Moreover, it is uniformly bounded by construction, namely $\xi_t \leq 2^p$. Thus, by Doob's supermartingale convergence theorem, the convergence $\xi_t \to 0$ holds $\Omega$-almost surely. Hence, using the pointwise comparison $\dist(x, y) \leq \frac{\pi}{2^p}(1-z)^p$ from Lemma \ref{lem:comparison}, we conclude that
$\dist(X_t, Y_t) \to 0$ almost surely for every $X_0, Y_0 \in \bbS^{n-1}$.

\emph{Step 2: Random attractor.} By Proposition \ref{prop:ExWeakPntAttr}, there exist the sample measures $\mu_\omega$ and a weak point attractor $A(\omega)$. Since the sphere is a compact manifold and the RQF is ergodic, by \cite[Proposition 2.6]{baxendale1991statistical}:
\[
(P_t)^*(\bar \rho \times \bar \rho) \to
\bbE(\mu_\omega(\rmd x) \times \mu_\omega(\rmd y))
\]
where $P^*_t$ is the adjoint of the semigroup of the two-point process \eqref{eq:rqf-two}. At the same time, for the test function $f(x, y) = \dist(x, y)$ using $\dist(X_t, Y_t) \to 0$ we obtain
\[
\lim_{t\to \infty}\int P_t f(x, y) \rmd(\bar \rho \times \bar \rho ) =  0 = \int \dist(x, y)\bbE (\mu_\omega(\rmd x) \times \mu_\omega(\rmd y)).
\]
and thus $\mu_\omega$ is $\Omega$-a.s. supported on a single point, which we denote by $a(\omega)$. Applying Proposition \ref{prop:discrete}, we conclude that the set 
$A(\omega) :=\{a(\omega)\}$ is the minimal weak point attractor of the RQF with forcing and $a(\omega)$ is measurable with respect to the past.

\emph{Step 3: Almost sure convergence.} First note that by the $\varphi$-invariance of the random attractor we necessarily have $A(\theta_t\omega) = \{a(\theta_t\omega)\} = \{\varphi(t, \omega, a(\omega))\}$. Thus, to upgrade to an almost sure convergence we argue as follows. Consider the Markov RDS $\hat \varphi$ corresponding to the two-point process of the RQF: 
\[
\hat \varphi(t, \omega, (x, y)) = (\varphi(t, \omega, x), \varphi(t, \omega, y)),
\]
and the function $\psi: C(\bbR_+, \bbS^{n-1} )\times C(\bbR_+, \bbS^{n-1} ) \to \{0,1\}$ defined as
\[
\psi(u, v) := 1\{\dist(u(t), v(t)) \to 0\},
\]
with the corresponding function $\Psi: \bbS^{n-1}\times \bbS^{n-1} \to [-1, 1]$:
\[
\Psi(x, y):= \bbE(\psi(\varphi(\cdot, \omega, x), \varphi(\cdot, \omega, y))) = \bbP\left[\dist(X_t, Y_t) \to 0\Big| X_0 = x, Y_0 = y\right].
\]
From step 1 we know that $\Psi((x, y)) \equiv 1$. Consider the random $\calF_{-\infty, 0}$-measurable point $\beta(\omega) = (x_0, a(\omega))$, then applying freezing Lemma \ref{lem:freezing} to the RDS $\hat \varphi$, we deduce
\[
\begin{aligned}
\MoveEqLeft \bbP[\dist(\varphi(t, \omega, x_0), \varphi(t, \omega, a(\omega)) \to 0]= \bbE[\Psi(\beta(\omega))] \geq \inf_{x, y} \Psi(x, y) = 1.
\end{aligned}
\]

\emph{Step 4: Convergence rate.} Arguing analogously to step 4 we obtain the expected rate of convergence to the random attractor. In particular, consider the same $\hat \varphi$ and $\beta(\omega)$ as in the previous step and define a family of functions $\kappa_t: C(\bbR_+, \bbS^{n-1}) \times C(\bbR_+, \bbS^{n-1})$ for $t\in \bbR_+$:
\[
\kappa_t (u, v)= \dist (u(t), v(t)),
\]
which is bounded by construction. The correspondning function $\Psi$ then takes the form
\[
\begin{aligned}
\Psi(x, y)&:=\bbE[\kappa_t(\hat \varphi(\cdot, \omega, (x, y)))] = \bbE[\kappa_t(\varphi(\cdot, \omega, x), \varphi(\cdot, \omega, y))] 
\\&= \bbE\left[\dist(X_t, Y_t) \Big| X_0 = x, Y_0 = y\right].
\end{aligned}
\]
Then, applying the freezing Lemma \ref{lem:freezing} and using the exponential convergence from Lemma \ref{lem:lyap-gamma} we obtain
\[
\begin{aligned}
 \bbE[\dist(\varphi(t, \omega, x_0), \varphi(t, \omega, a(\omega))]= \bbE[\Psi(\beta(\omega))] \leq \sup_{x, y} \Psi((x, y)) = \pi e^{-\lambda_\gamma t},
\end{aligned}
\]
where the bound 
\[
\sup_{x, y} \Psi((x, y)) \leq \sup_{z\in[-1, 1]}(1-z)^p\frac{\pi}{2^p} e^{-\lambda_\gamma t} = \pi e^{-\lambda_\gamma t}.
\]
follows from Lemmas \ref{lem:lyap-gamma} and \ref{lem:comparison}.
\end{proof}

\begin{remark}[Contraction rate] We remark that the top Lyapunov exponent classically provides local exponential convergence rate for points in the $\omega$-dependent neighborhood of the attractor; namely bounds of the form
\[
\dist(\varphi(t, \omega, x), A(\theta_t\omega)) \leq C(\omega, x)e^{\Lambda_{\max}t}\dist(x, A(\omega)), \quad \forall x:\dist (x, A(\omega)) < r(\omega, x),
\]
which do not generally imply exponential convergence to the attractor in expectation. Notably, for the classical Brownian flow on $\bbS^{n-1}$ given by 
\begin{equation}
\label{eq:bm}
\rmd X_t = P_{X_t}\partial W_t,
\end{equation}
\cite[Theorem 5.3(ii)]{baxendale1986asymptotic} establishes almost sure exponential convergence with the rate $e^{-\frac{1}{2}(n-1) t}$. We also remark that our rate $-\lambda_\gamma/\gamma^2$ recovers the result of \cite{baxendale1986asymptotic}
\[
-\frac{\lambda_\gamma}{\gamma^2} = -\frac{2(n-1)}{4+\gamma^2} \to -\frac{1}{2}(n-1)
\]
in the limit $\gamma \to 0$. 
\end{remark}
\begin{remark}[Rate for large $\gamma$]
\label{rem:large-gamma}
For large $\gamma$ the second branch of $\alpha(\gamma, n, p)$ becomes optimal in Lemma \ref{lem:lyap-gamma}, giving the optimal value $\hat p = \frac{\gamma^2(n-1) +4}{4(\gamma^2 +4)}$. At the same time, for $n\geq 4$ and large enough $\gamma^2$ we have $\hat p \geq \frac{1}{2}$. Thus we can take $p =\frac{1}{2}$ directly estimating $\sqrt{1-z}$ to get the contraction rate
\[
\lambda_\gamma = p\alpha(p, \gamma, n)= \frac{1}{2}(\gamma^2(n-1) + 4-\gamma^2 - 4) = \frac{(n-2)}{2}\gamma^2
\]
for $\gamma \gg 1$.
\end{remark}

\begin{remark}[Generalizations] 
\label{rem:generalization}
We expect similar estimates to hold for a larger class of isotropic diffusions on $\bbS^{n-1}$. In particular, for the classical formulation of a spherical Brownian motion \eqref{eq:bm} the proof applies in a simplified form since the $\gamma$-independent components are not present. In particular, taking $p \uparrow \frac{1}{2}$ and following the steps of the proof above, yields the following estimate
\[
\bbE\dist(\tilde \varphi(t, \omega, x_0), A(\theta_t\omega)) \leq C_0e^{-\frac{1}{2}(n-2) t},
\]
which is consistent with the rate for the large $\gamma$ regime in Remark \ref{rem:large-gamma}.
We also note that the rate in expectation is weaker than the almost sure rate $e^{-\frac{1}{2}(n-1) t}$ established in \cite[Theorem 5.3(ii)]{baxendale1986asymptotic}. 

To cover the case $n = 2$, it is also possible to choose $p\in (0, \frac{1}{2})$ still giving exponential contraction and therefore almost sure convergence. For example, taking $p = \frac{1}{4}$ we get the estimate 
\[
\bbE\dist(\tilde \varphi(t, \omega, x_0), A(\theta_t\omega)) \leq C_0e^{-\frac{1}{8} t}.
\]
\end{remark}  

\subsection{Metastable synchronization}
\label{sec:main-meta}
Finally, we consider the small forcing regime $\gamma \to 0$ and establish the multi-scale behaviour of the two-point process \eqref{eq:rqf-two}. In this Section we slightly change the notation and use the upper index $\gamma$ to denote the RQF process with the forcing $\gamma$, namely $X_t^\gamma$ and, analogously, $\varphi^\gamma$ for the corresponding RDS. In particular, $X_t^0$ denotes the RQF without forcing as studied in \cite{engel2026random}. Since the dynamics for small $t$ is dominated by quadratic component, the proof relies on a coupling argument between the processes $X_t^\gamma$ and $X_t^0$. 

First, arguing analogously to the Theorem \ref{th:main}, we establish convergence rates for the non-forced RQF $\varphi^0$.
\begin{proposition}
    [Exponential convergence for $\gamma = 0$]
\label{lem:gamma-0}
Let $\gamma = 0$, then there exists a random set $A^Q(\omega) = A^Q(\omega^Q)$ consisting of two anti-polar points
\[
A^Q(\omega) = \{a^Q(\omega), -a^Q(\omega)\},
\]
where $a^Q(\omega)$ is an $\calF^Q_{-\infty, 0}$-measurable random point and the RQF RDS satisfies:
\begin{equation}
\label{eq:estimate-two}
\bbE(\dist(\varphi^0(t, \omega, x_0), A^Q(\theta_t\omega))) \leq \frac{\pi}{\sqrt{2}} e^{-\frac{1}{2}t},
\end{equation}
for all $x_0 \in \bbS^{n-1}$
\end{proposition}
The proof of Proposition \ref{lem:gamma-0} is analogous to the proof of Theorem \ref{th:main} and is postponed to Section \ref{sec:proofs}. We are now moving to the main result.

\begin{theorem}[Metastability of the two-point motion]
\label{th:main-meta}
For any $\gamma \in \bbR$ there exists a random set 
\[
A^Q(\omega^Q) = \{a^Q(\omega^Q), - a^Q(\omega^Q)\},
\]
where $a^Q(\omega^Q)$ is an $\calF^Q_{-\infty, 0}$-measurable function,
and the RQF RDS $(\varphi^\gamma, \theta)$ satisfies 
    \begin{equation}
     \label{eq:estimate-meta}   
    \bbE \dist(\varphi^\gamma(t, \omega, x_0), A^Q(\theta_t\omega)) \leq \frac{\pi}{\sqrt{2}} e^{-\frac{1}{2}t} + \frac{\pi\sqrt{n-1}}{4}|\gamma| \sqrt{e^{4t}-1},
    \end{equation}
    for all $x_0 \in \bbS^{n-1}$.
\end{theorem}
\begin{remark}[The time scale]
    The two terms of the upper bound \eqref{eq:estimate-meta} balance at $t \in \bbR_+$ solving
    \[
    e^{-\frac{1}{2}t} = \frac{\sqrt{n-1}}{4}|\gamma| \sqrt{e^{4t}-1},
    \]
    which corresponds to 
    \[
    t =\frac{2}{5}\log |\gamma|^{-1} -\frac{1}{5}\log n +O(1)
    \]
    as $\gamma \to 0$. This estimate supports the statement of the 'meta'-attractor $A^Q$ being attractive on the time scale of order $\log |\gamma|^{-1}$.
\end{remark}
\begin{proof}
    We consider the coupling of the RQFs with and without forcing $(\varphi^0, \varphi^\gamma)$ and by triangle inequality obtain
    \[
    \begin{aligned}
    \bbE \dist(\varphi^\gamma(t, \omega, x_0), A^Q( \theta_t\omega)) &\leq \bbE \dist(\varphi^0(t, \omega, x_0), A^Q( \theta_t\omega)) \\
    &\qquad+ \bbE \dist(\varphi^0(t, \omega, x_0), \varphi^\gamma(t, \omega, x_0)) := A(t) + B(t).
    \end{aligned}
    \]
    By Proposition \ref{lem:gamma-0}, the first term decays exponentially $A(t) \leq \frac{\pi}{\sqrt{2}}e^{-\frac{1}{2}t}$, and thus we only need to bound the distance between $\varphi^0$ and $\varphi^\gamma$. 
    
    Let $V_t := X^0_t - X_t^\gamma, \ V_0 = 0$, then $V_t$ solves the following SDE
    \begin{equation}
 \label{eq:vt}
    \rmd V_t = \frac{n-1}{2}\left(\gamma^2X^\gamma_t-V_t \right)\rmd t -\gamma P_{X_t^\gamma}\rmd W_t + (P_{X_t^0}\rmd Q_t X_t^0 -P_{X_t^\gamma}\rmd Q_t X_t^\gamma),
    \end{equation}
    and applying Ito formula to $f(x) = \|x\|^2$ we conclude that the evolution of $\|V_t\|^2$ follows the dynamics
    \begin{equation}
    \label{eq:th-meta-v}
    \rmd\|V_t\|^2 = (\gamma^2(n-1) +\lambda_V\|V_t\|^2)\rmd t+ \rmd M_t,
   \end{equation}
    where 
    \[
    \lambda_V = 2\left<X_t^\gamma, X_t^0\right>^2 +2\left<X_t^\gamma, X_t^0\right> -\frac{\gamma^2(n-1)}{2} \leq 4,
    \]
    and $\bbE M_t =0$ is a martingale, see Lemma \ref{lem:vt} for details.
    Integrating Eq. \eqref{eq:th-meta-v} and taking expectation on both sides then gives
    \[
    \bbE \|V_t\|^2 \leq \|V_0\|^2 +\int_0^t 4\bbE\|V_s\|^2 \rmd s +\gamma^2(n-1)t.
    \]
     Applying standard Gronwall argument we obtain the upper bound
     \[
     \bbE \|V_t\|^2 \leq \frac{\gamma^2(n-1)}{4}(e^{4t} -1) \leq \frac{\gamma^2(n-1)}{4}(e^{4t} -1).
     \]
     Finally, note that by Cauchy-Schwartz inequality $B(t) \leq \sqrt{\bbE \|V_t\|^2}$. Therefore, using $\dist(x, y) \leq \frac{\pi}{2}\|x-y\|$ and combining the estimates for $A(t)$ and $B(t)$ we get the result.    
    \end{proof}
    
\subsection{Auxiliary lemmas}
\label{sec:proofs}
We begin with characterizing the dynamics of the process $Z_t$, which is used to establish the moment bound in Lemma \ref{lem:lyap-gamma} and Proposition \ref{lem:gamma-0}.
\begin{lemma}[Dynamics of $Z_t$] 
\label{lem:z}
Let $(X_t, Y_t)$ be the two-point motion of the RQF with forcing and let $Z_t = \left<X_t, Y_t\right>$. Then $Z_t$ solves the Ito stochastic differential equation
\[
\begin{aligned}
\rmd Z_t &= \left(\gamma^2(1-Z_t)(n-2-Z_t) -2Z_t(1-Z_t^2) \right)\rmd t \\
&\qquad - Z_t(X_t^T\rmd Q_t X_t + Y_t^T\rmd Q_t Y_t) + 2X_t^T\rmd Q_t Y_t + \gamma(1-Z_t)(X_t + Y_t)^T\rmd W_t.
\end{aligned}
\]
In particular, the infinitesimal generator of the process $Z_t$, takes the form
\begin{align}
(L_\gamma f)(z)  &= (1-z)\left( \gamma^2(n-2 - z) - 2z(1+z)\right)\partial_z f \notag \\
&\qquad+ (1-z)^2(1+z)(2(1+z) +\gamma^2)\partial_{zz}f, \label{eq:generator}
\end{align}
for any $f\in C^{\infty}([-1, 1])$. 
\end{lemma} 
\begin{proof}
    We start with reformulating the RQF \eqref{eq:rqf} in the Ito form:
    \[
    \rmd X_t = -\frac{(n-1)(1+\gamma^2)}{2}X_t\rmd t + P_{X_t}\rmd Q_tX_t + \gamma P_{X_t}\rmd W_t,
    \]
    which follows from \cite[Lemma 4.9]{engel2026random} and the classical representation of a Brownian motion on $\bbS^{n-1}$.
    Then, applying Ito's lemma to the function $f(x, y) = \left<x, y\right>$ we obtain
    \begin{align}
 \rmd Z_t &= \rmd f(X_t, Y_t)= X^T_t \rmd Y_t + Y^T_t\rmd X_t + \rmd [X_t, Y_t]  \notag \\
 &= -(n-1)(1+\gamma^2 )Z_t \rmd t + \rmd [X_t, Y_t] + \rmd M_t^\gamma, \label{eq:lem32}
\end{align}
 where the martingale $M_t^\gamma$ is given by the following equation
 \[
 \rmd M_t^\gamma = -Z_t(X_t^T\rmd Q_t X_t + Y_t^T\rmd Q_t Y_t)+ 2X_t\rmd Q_t Y_t + \gamma(1-Z_t)(X_t + Y_t)^T\rmd W_t.
 \]
 Calculating the quadratic covariation $[X_t, Y_t] = \int_0^t q(Z_s) \rmd s$ we obtain the following expression for the function $q$:
 \[
 \begin{aligned}
 q(z) &= q_{\gamma=0}(z) + \gamma^2 \sum_{i, j} P_x^{i, j}P_y^{i, j} =q_{\gamma=0}(z) + \gamma^2\sum_{i, j}(\delta_{i, j} - x^ix^j)(\delta_{i, j} - y^iy^j) \\
 &=nz - 3z +2z^3 + \gamma^2(n +z^2 - 2),
  \end{aligned}
 \]
 where 
 \[
 q_{\gamma=0}(z) = nz - 3z +2z^3
 \]
 is the quadratic covariation of the quadratic noise component of the RQF as follows from the proof of \cite[Lemma 4.9]{engel2026random}. Plugging the expression into \eqref{eq:lem32} we obtain the cumulative drift
 \[
 -(n-1)(1+\gamma^2)z + nz - 3z +2z^3 + \gamma^2(n +z^2 - 2) = \gamma^2(1-z)(n-2 - z) -2z(1-z^2).
 \]
 Finally, to obtain \eqref{eq:generator} we calculate the covariance of the process $Z_t$:
 \[
 \begin{aligned}
 \Sigma(z) &= \Sigma_{\gamma = 0}(z) + \gamma^2\sum_{i}(1-z)^2(x^i + y^i)^2 =4(1 - z^2)^2 + 2\gamma^2(1-z)^2(1 + z)
 \end{aligned}
 \]
 where $\Sigma_{\gamma = 0}(z) =4(1 - z^2)^2$ is the covariance of the quadratic part derived in \cite[Theorem 4.8]{engel2026random}, and hence the result.
\end{proof}

We are now ready to prove Lemma \ref{lem:lyap-gamma} and Proposition \ref{lem:gamma-0}.

\begin{proof}[Proof of Lemma \ref{lem:lyap-gamma}]
    By Lemma \ref{lem:z}, the generator of the process $Z_t$ takes the closed form and is therefore decoupled from the dynamics of $(X_t, Y_t)$. We consider the test function $g_{p, \delta} = (1-z +\delta)^p$ for $\delta, p \in (0, 1)$. Note that $g_{p, \delta}$ is smooth on $[-1, 1]$ and since the process is restricted to the interval $[-1, 1]$ the discontinuity of $g'_{p, \delta}$ at $z = 1+\delta$ does not play any role. For $g_{p, \delta}$ we calculate
\[
\begin{aligned}
\MoveEqLeft (L_\gamma g_{p, \delta})(z)= -p(1-z)\left( \gamma^2(n-2 - z) - 2z(1+z)\right)(1-z +\delta)^{p-1} \\
&\quad + p(p-1)(1-z)^2(1+z)(2(1+z) +\gamma^2)(1-z+\delta)^{p-2} \\
&:=-p\Lambda_p(n, \gamma, z)g_{p, \delta} + R_p(\delta, \gamma, z),
\end{aligned}
\]
where 
\[
\begin{aligned}
\Lambda_p(n, \gamma, z) &= \gamma^2(n-2 - z) - 2z(1+z) - (p-1)(1+z)(2(1+z) +\gamma^2) \\
&=\gamma^2(n-1) + (z+1)\left(2 - p(\gamma^2 + 2(1+z))\right) , \\
R_p(n, \delta, \gamma, z) &= p\delta(1-z+\delta)^{p-1}(\gamma^2(n-2-z)-2z(1+z)) \\
&\qquad - p(p-1)(1-z+\delta)^{p-2} (1+z)(2(1+z) +\gamma^2)\delta(\delta + 2(1-z))\\
&=:A + B
\end{aligned}
\]
Let $s = 1+z$ and rewrite $\Lambda_p$ as
\[
\Lambda_p(n, \gamma, z) = \gamma^2(n-1) + s(2-p(\gamma^2 + 2s)).
\]
Since it is a concave function of $s$ for any given interval it is minimized at one of the end points. Since $z\in [-1, 1]$, and therefore $s\in [0, 2]$ we conclude that
\[
\min_{z\in [-1, 1]} \Lambda_p(n, \gamma, z) =\min\{\gamma^2(n-1), \gamma^2(n-1)+ 4 - 2p(\gamma^2 + 4))\} := \alpha(\gamma, n, p),
\]
implying that for every $p \in(0, p^*)$, where $p^* = \frac{\gamma^2(n-1)+4}{2(\gamma^2 + 4)}$, $\alpha(\cdot, \cdot, p)$ is a strictly positive function.  
At the same time, using $z\in [-1, 1]$, $p\in (0, p^*)$ and $p^*<1$ we obtain the following upper bounds:
\[
\begin{aligned}
    |A| &\leq \delta(1-z+\delta)^{p-1}(\gamma^2(n-1)+4) \leq \delta^p(\gamma^2(n-1) +4) \\
    |B| &\leq \delta^2(1-z+\delta)^{p-2} 2(4 +\gamma^2) \leq 2\delta(1-z+\delta)^{p-1}(4 +\gamma^2) \leq 2\delta^p(4 +\gamma^2),
\end{aligned}
\]
hence,
\[
|R_p(n, \delta, \gamma, \cdot )| \leq \delta^p(\gamma^2n +12) := \delta^p\beta(\gamma, n) 
\]
on $[-1, 1]$.
Applying Dynkin's formula (Proposition \ref{prop:dynkin}) to $g_{p, \delta}$ for all $\delta>0$ we obtain
\[
\bbE g_{p, \delta}(Z_t) \leq g_{p, \delta}(z_0) +\bbE\left(\int_0^t (-p\alpha(\gamma, n, p)g_{p, \delta}(Z_s) + \delta^p\beta(\gamma, n))\rmd s\right).
\]
Using a standard variation of constants estimate we therefore conclude that
\[
\bbE g_{p, \delta}(Z_t) \leq g_{p, \delta}(z_0)e^{-p\alpha(\gamma, n, p)t} + \delta^p\frac{\beta(\gamma, n)}{p\alpha(\gamma, n, p)}.
\]
Since the inequality $g_{p, 0}(z) \leq g_{p, \delta}(z)$ holds for all $\delta >0$ pointwise on $[-1, 1]$, we can bound the $p$-th moment of $(1-Z_t)$ by:
\[
\bbE (1-Z_t)^p \leq \bbE (1-Z_t +\delta )^p \leq g_{p, \delta}(z_0)e^{-p\alpha(\gamma, n, p)t} + \delta^p\frac{\beta(\gamma, n)}{p\alpha(\gamma, n, p)}
\]
Taking limit $\delta \to 0$ concludes the proof.
\end{proof}

\begin{proof}[Proof of Proposition \ref{lem:gamma-0}]
    First note that the existence of the strong forward point attractor $A^Q(\omega)$ follows directly from \cite[Theorem 4.8]{engel2026random}. The proof of the estimate \eqref{eq:estimate-two} follows the structure of the proof of Theorem \ref{th:main} with the Lyapunov function $h_{p, \delta}(z) = (1-z^2 +\delta)^p$. In particular the generator of $Z_t^0$ takes the form
    \[
    (L_0f)(z) = -2z(1-z^2)\partial_z f(z) +2(1-z^2)^2\partial_{zz}f(z),
    \]
    for all $f\in C^\infty([-1, 1])$. Plugging in $h_{p, \delta}$ we obtain the expression
    \[
    \begin{aligned}
    L_0h_{p, \delta} &= 4pz^2(1-z^2)(1-z^2+\delta)^{p-1} \\
    &\qquad+4p(1-z^2)^{2}(1-z^2+\delta)^{p-2}((p-1)2z^2 -(1-z^2 +\delta)) \\
    &=4p (1-z^2+\delta)^{p}(z^2 + (p-1)2z^2 - (1-z^2)) \\
    &\qquad +4p\delta(1-z^2+\delta)^{p-2}\left((1-2z^2)(1-z^2+\delta)+2(1-p)z^2[2(1-z^2)+\delta]\right)\\
    &:=-4p \hat \Lambda_p(z) h_{p, \delta} + \hat R_p(z, \delta),
    \end{aligned}
    \]
    where
    \[
    \begin{aligned}
        \hat \Lambda_p(z) &= (1-2pz^2)\\
       \hat R_p(z, \delta) &=4p\delta(1-z^2+\delta)^{p-2}\left((1-z^2)(1+2(1-2p)z^2) + \delta(1-2pz^2)\right),
    \end{aligned}
    \]
    and for $p\in (0, \frac{1}{2}]$
    \[
    |\hat R_p| \leq 4p\delta(1-z^2+\delta)^{p-2}(3(1-z^2) +\delta) \leq 12p\delta^p.
    \]
    Choosing $p = \frac{1}{4}$ we get $\hat \Lambda_{\frac{1}{4}}(z) \geq \frac{1}{2}$ and reproducing the variation of constants estimate from Lemma \ref{lem:lyap-gamma} we bound the growth of the random variable $(1-Z_t^2)^{\frac{1}{4}}$ by 
    \[
    \bbE (1-Z_t^2)^\frac{1}{4} \leq (1-z_0^2)^\frac{1}{4}e^{-\frac{1}{2}t}.
    \]
    Since the points converge to either polar or anti-polar configuration, we define
    \[
    \xi_t := \min(\|X_t - Y_t\|, \|X_t+ Y_t\|),
    \]
    and, using $\|X_t \pm Y_t\| = \sqrt{2(1\pm z)}$, obtain
    \[
    \xi_t \leq \sqrt{\|X_t +Y_t\|\|X_t - Y_t\|} = \sqrt[4]{4(1-Z_t)(1+Z_t)} = \sqrt{2}(1-Z^2_t)^{\frac{1}{4}}, 
    \]
    which immediately yields 
    \[
    \bbE \xi_t \leq \sqrt{2}e^{-\frac{1}{2}t}.
    \]
    Similarly to the proof of Theorem \ref{th:main}, applying the freezing argument and using $\dist(x, y) \leq \frac{\pi}{2}\|x-y\|$ we conclude that
    \[
    \bbE \dist(\varphi(t, \omega, x_0), A^Q(\theta_t\omega)) \leq \frac{\pi}{2}\bbE \xi_t  \leq \frac{\pi}{\sqrt{2}}e^{-\frac{1}{2}t}.
    \]
    Where
    \[
    A^Q(\theta_t\omega)= \{\varphi^0(t, \omega, a^Q(\omega)), -\varphi^0(t, \omega, a^Q(\omega))\}
    \]
    by invariance of the random attractor.
\end{proof}
Finally, we derive the dynamics of the process $\|V_t\|^2$ used in the proof of Theorem \ref{th:main-meta}.
\begin{lemma}[Dynamics of $\|V_t\|^2$]
 \label{lem:vt}   
 Consider $X_t^0, X_t^\gamma$ and $V_t$ as in Eq. \eqref{eq:vt}, and let $U_t = \|V_t\|^2$, then $U_t$ solves
    \[
    \rmd U_t = b(X_t^0, X_t^\gamma, U_t)\rmd t + \rmd M_t,
    \]
    where $M_t$ is a martingale with $\bbE M_t =0$ and
    \[
    b(X_t^0, X_t^\gamma, U_t) = \gamma^2(n-1) + U_t\left(2\left<X_t^\gamma, X_t^0\right>^2 +2\left<X_t^\gamma, X_t^0\right> -\frac{\gamma^2(n-1)}{2}\right).
    \]
\end{lemma}
\begin{proof}
Applying Ito's lemma to $f(x) = \|x\|^2 =\left<x, x\right>$ we obtain
\[
\rmd f(V_t) = 2\left<V_t,\rmd V_t\right> + \rmd [V_t, V_t]
\]
The first term gives
\[
\begin{aligned}
2\left<V_t,\rmd V_t\right> = (n-1)(\gamma^2 \left<V_t, X_t^\gamma\right> - U_t)\rmd t + \rmd M_t = -(n-1)U_t\left(1+\frac{1}{2}\gamma^2\right)\rmd t + \rmd M_t,
\end{aligned}
\]
where we used
\[
\left<V_t, X_t^\gamma\right> = \left<X_t^0- X_t^\gamma, X_t^\gamma\right> = \left<X_t^0, X_t^\gamma\right> - 1 = -\frac{1}{2}\left<X_t^0- X_t^\gamma, X_t^0- X_t^\gamma\right> = -\frac{1}{2}\|V_t\|^2.
\]
Calculating the quadratic covariation $[V_t, V_t] = \int_0^t q(s) \rmd s$  we obtain
\[
\begin{aligned}
q(s) &= \frac{1}{2}\sum_{i, j, k}\left(P_{X_s^\gamma}^{i, j}(X_s^\gamma)^k - P_{X_s^0}^{i, j}(X_s^0)^k + P_{X_s^\gamma}^{i, k}(X_s^\gamma)^j - P_{X_s^0}^{i, k}(X_s^0)^j\right)^2 + \gamma^2 \sum_{i, j} (P_{X_s^\gamma}^{i, j})^2\\
&:= I + II.
\end{aligned}
\]
The second term is essentially the trace of the projection matrix and therefore $II = \gamma^2(n-1)$.

Let $\sigma_{i, j}^k(x) := \frac{1}{\sqrt{2}}(P_{x}^{i, j}x^k + P_{x}^{i, k}x^j)$, then, simplifying the first term, we obtain
\begin{equation}
\label{eq:vt-1}
I = \sum_{i, j, k}(\sigma_{i, j}^k(X_t^\gamma))^2 + \sum_{i, j, k}(\sigma_{i, j}^k(X_t^0))^2 - 2\sum_{i, j, k}\sigma_{i, j}^k(X_t^0)\sigma_{i, j}^k(X_t^\gamma),
\end{equation}

where for all $a, b \in \bbS^{n-1}$
\[
\sum_{i, j, k}\sigma_{i, j}^k(a)\sigma_{i, j}^k(b) = (n-3)\left<a, b\right> +2\left<a, b\right>^3.
\]
as follows from the proof of \cite[Lemma 4.9]{engel2026random}. In particular, for $a = b$ we obtain
\[
\sum_{i, j, k}(\sigma_{i, j}^k(a))^2 = (n-3) + 2 = n-1.
\]
Plugging it into \eqref{eq:vt-1} and using $\|V_t\|^2 = 2- 2\left<X_t^0, X_t^\gamma\right>$ we obtain
\[
\begin{aligned}
I &= 2(n-1) - 2(n-3)\left<X_t^0, X_t^\gamma\right> - 4\left<X_t^0, X_t^\gamma\right>^3 \\
&=\|V_t\|^2((n-1) +2\left<X_t^0, X_t^\gamma\right>^2 + 2\left<X_t^0, X_t^\gamma\right>),
\end{aligned}
\]
and hence the total drift takes the form
\[
\begin{aligned}
   b(X_t^0, X_t^\gamma, U_t) &= \gamma^2(n-1)-(n-1)U_t\left(1+\frac{1}{2}\gamma^2\right) + U_t((n-1) +2\left<X_t^0, X_t^\gamma\right>^2 + \left<X_t^0, X_t^\gamma\right>) \\
   &= \gamma^2(n-1)+ U_t \left(2\left<X_t^\gamma, X_t^0\right>^2 +2\left<X_t^\gamma, X_t^0\right> -\frac{\gamma^2(n-1)}{2}\right).
\end{aligned}
\]
Finally note that the process $M_t$ satisfies $M_0 = 0$ and is given by
\[
    \rmd M_t = \sum_i h^i_t \rmd W_t^i + \sum_{i, j} g^{i, j}_t \rmd B_t^{i, j}, \qquad |h^i_t|, |g^{i, j}_t| < C,
    \]
    where all $h^i_t, g^{i, j}_t$ are adapted processes. Thus $M_t$ is a martingale and $\bbE M_t =0$. 
\end{proof}
\section{Multi-cluster and multiscale dynamics}
\label{sec:multi}
\begin{figure}[t]
    \centering
    \begin{subfigure}{0.43\textwidth}
    \includegraphics[width=0.49\linewidth]{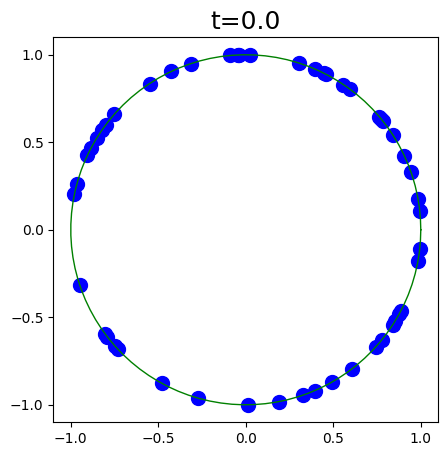}
    \includegraphics[width=0.49\linewidth]{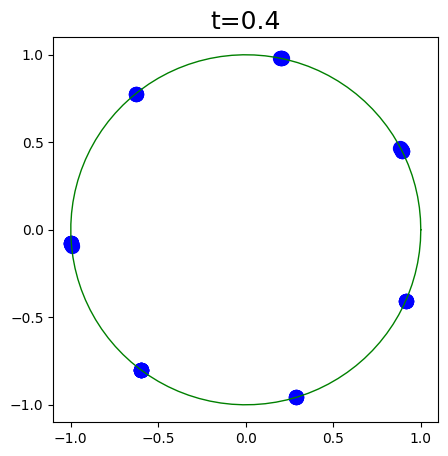}
    \includegraphics[width=0.49\linewidth]{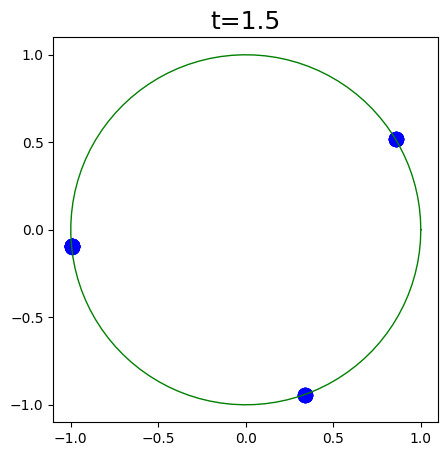}
    \includegraphics[width=0.49\linewidth]{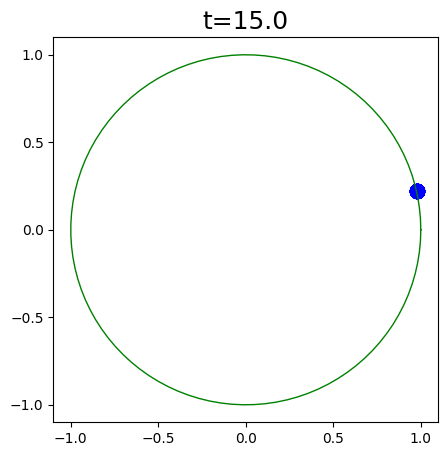}
    \end{subfigure}
    \begin{subfigure}{0.56\textwidth}
    \includegraphics[width=\linewidth]{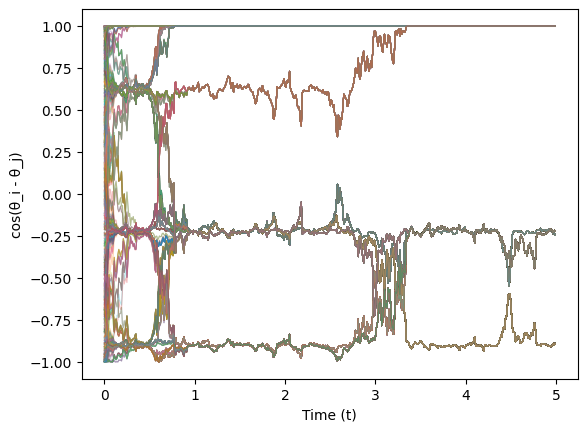}
    \end{subfigure}
    \caption{On the left: the ensemble of solutions of Eq. \eqref{eq:multi} driven by the same noise from different initial conditions. The system gradually approaches each of the random attractors for $k = 7, 3$ and $1$. On the right: the dynamics of pairwise scalar products between the particles showing the formation and dissolvement of the clusters.}
    \label{fig:multi}
\end{figure}
In this section we discuss how the same approach can be extended to the cases of multi-clustered attractors and multi-scale dynamics. From now on we only consider the dynamics on the circle, namely the case $n =2$. Formulating a Brownian motion on a sphere of $n>2$ dimensions with a discrete random attractor consisting of $k> 2$ points is generally an open question.

\subsection{Multiple clusters}
To construct a system with a random attractor consisting of $k>1$ points we consider the angular representation of the RQF on the circle as in \cite[Section 5.1]{engel2026random}. In particular, consider the Stratonovich SDE with $k$-harmonic coefficients of the form
\begin{equation}
\label{eq:multi-k}
\rmd\phi_t = \cos(k\phi_t)\partial B^1_t + \sin(k\phi_t)\partial B^2_t,
\end{equation}
where $B^{1,2}_t$ are independent Brownian motions. Notice that for any $k\in \bbN$ the variable $\psi_t = k\phi_t$ satisfies the equation
\begin{equation}
\label{eq:multi-1}
\rmd\psi_t = k\cos \psi_t\partial B^1_t + k\sin \psi_t \partial B^2_t.
\end{equation}
The model \eqref{eq:multi-1} is the angular representation of the (rescaled) classical Brownian motion as in Eq. \eqref{eq:bm} on $\bbS^1$, implying that so is \eqref{eq:multi-k}. Moreover, the RDS corresponding to \eqref{eq:multi-1} is fully synchronizing, namely its random point attractor is almost surely a singleton \cite[Section 5]{baxendale1986asymptotic}. Therefore, the random attractor of the $k$-harmonic model \eqref{eq:multi-k} consists of at most $k$ points and the following extension of \cite[Proposition 5.1]{engel2026random} holds.
\begin{proposition}[Harmonic model on $\bbS^{1}$]
\label{prop:harmonic}
Let $k\in\bbN$ and $\phi_t$ be the harmonic model on $\bbR / 2 \pi \bbZ \cong \bbS^1$ defined in  Eq.~\eqref{eq:multi-k}, then 
\begin{itemize}
    \item[(i)] $\phi_{t}$ is a Brownian motion
    \item[(ii)] the minimal weak random point attractor $A(\omega)$ of the corresponding RDS is supported on exactly $k$ equidistant $\calF_{-\infty, 0}$-measurable random points:
\[
A(\omega) = \left\{a(\omega)+\frac{2m\pi}{k}: m = 0\dots k-1\right\}.
\]
\end{itemize}
\end{proposition}
\begin{proof}
    The infinitesimal generator of the process \eqref{eq:multi-1} is $L = \frac{k^2}{2}\partial^2_{\psi\psi}$, changing the variables back to $\phi$ gives \emph{(i)}. To show \emph{(ii)} note that the proof of \cite[Proposition 5.1]{engel2026random} applies with the reparametrization $\psi_t = k\phi_t$.
\end{proof}
At the same time, according to the Remark \ref{rem:generalization}, we expect the model \eqref{eq:multi-1} to satisfy at least
\[
\bbE\dist(\psi(t, \omega, x_0), A(\theta_t\omega)) \leq C_0e^{-\frac{k^2}{8} t}.
\]
The fact $\dist(\phi^1, \phi^2) = k^{-1}\dist(\psi^1, \psi^2)$ therefore implies exponential convergence to the random attractor in expectation for the $k$-harmonic model at the same rate $e^{-\frac{k^2}{8} t}$.
 \subsection{Multiple time-scales}
To construct a model with multiple time scales note that the RQF with forcing in angular coordinates is written as
\begin{equation*}
\rmd\phi_t = (\cos(2\phi_t)\partial B^1_t + \sin(2\phi_t)\partial B^2_t) + \gamma (\cos(\phi_t)\partial B^3_t + \sin(\phi_t)\partial B^4_t),
\end{equation*}
where the harmonic noise with $k=1$ is treated as perturbation of a harmonic model with $k=2$. We emphasize that the perturbation corresponds to the lower harmonic and that the multiscale dynamics would not be present in the opposite case, as follows from the estimates in Theorem \ref{th:intro-gamma} for $\gamma \gg 1$. 
Indeed, if a higher harmonic serves as a small perturbation, then the unperturbed system synchronizes to a singleton on the fast time-scale and the perturbation will not be able to split the mass into two points again.
Thus, to construct an model exhibiting a multiscale synchronization we consider a cascade of $m$ harmonic noises of a decreasing order, namely the SDE of form:
\begin{equation}
\label{eq:multi}
\rmd\phi_t = \sum_{i=0}^{m-1}\gamma^i(\cos(k_i\phi_t)\partial B^{i, 1}_t + \sin(k_i\phi_t)\partial B^{i, 2}_t), \quad k_i > k_{i+1},
\end{equation}
where $B^{i, 1},B^{i, 2}$ are independent Brownian motions. We first note that $\phi$ is a rescaled Brownian motion on the circle.
\begin{proposition}[Multiharmonic model is a Brownian motion]
    Let $\phi_t$ be the multiharmonic model as in Eq. \eqref{eq:multi}, then $\phi_{t/w(\gamma, m)}$ is a Brownian motion, where $w(\gamma, m) = \sum_{i=0}^{m-1} \gamma^{2i}$.
\end{proposition}
\begin{proof}
    As follows from Proposition \ref{prop:harmonic} and the independence of the driving processes $B^{i, 1},B^{i, 2}$, the generator of the multiharmonic process $\phi_t$ is $Lf = \frac{1}{2}\sum_{i=0}^{m-1} \gamma^{2i} \Delta$ and hence the result.
\end{proof}
We remark that the time scale appearing in the Fokker-Planck equation of the multiharmonic model does not depend on the harmonics $k_i$ but only on the number of components $m$ and the small parameter $\gamma$.  At the same time, the parameters $k_i$ define the structure of the attractive configuration. In particular, the model with $m$ harmonics admits $m-1$ 'meta'-attractors $A_i(\omega):\  i = 0\dots m-2$ and one global attractor $A_{m-1}(\omega)$ where the structure the $j$-th attractor depend on the harmonics $(k_i)_{i \leq j}$.

To show the existence of $m$ different random attractors and the corresponding $m$-scale behaviour we argue as follows. We define the first random attractor $A_0(\omega)$ as the attractor of $k_0$-harmonic model \eqref{eq:multi-k}. According to Proposition \ref{prop:harmonic}, it consists of exactly $k_0$ equidistant points and  is exponentially attractive with a rate independent on $\gamma$ as discussed above. By a coupling argument analogous to the proof of Theorem \ref{th:main-meta},  we also conclude that the divergence from $A_0(\omega)$ is linear in $\gamma$. 

To recover the second scale we define $A_1(\omega)$ as the random attractor of the truncated system driven by the first two harmonic noises:
\[
\rmd\phi^{1}_t = \sum_{i=0}^1\gamma^i(\cos(k_i\phi^1_t)\partial B^{i, 1}_t + \sin(k_i\phi^1_t)\partial B^{i, 2}_t).
\]
By Proposition \ref{prop:ExWeakPntAttr}, there exists a weak point attractor $A_1(\omega)$ of the truncated system bi-harmonic model $\phi^{1}_t$ and its structure only depends on the first two harmonic numbers $k_0$ and $k_1$. Iterating this construction and considering the next truncated model
\[
\rmd\phi^{2}_t = \sum_{i=0}^2\gamma^i(\cos(k_i\phi^2_t)\partial B^{i, 1}_t + \sin(k_i\phi^2_t)\partial B^{i, 2}_t).
\]
we obtain a cascade of random weak point attractors. 

We also expect that the explicit exponential rates of convergence to the random point attractors $A_i(\omega)$ are available when $k_{i} \mod k_{i+1} = 0$ for all $i <m$. In this case $A_i(\omega)$ consists of exactly $k_i$ equidistant points and, by analogy with the RQF, the Lyapunov function of the two-point process at the $i$-th time scale is the corresponding harmonic function $f(\xi^1, \xi^2) = (1- \cos(k_i(\xi^1 - \xi^2)))^p$. We conjecture that in this case an analog of Theorem \ref{th:intro-gamma} could be formulated iteratively for each truncation. 

At the same time, we expect the multiscale behavior to appear in a general multi-harmonic setting and we illustrate it on Figure \ref{fig:multi}. We show the dynamics of the multi-point motion of the system \eqref{eq:multi} with $m=3$ and $k = 7,3,1$. As expected, at three different scales the model concentrates around different number of clusters. It is important to note that the harmonics are not divisible and thus
the random attractor $A_1(\omega)$ is not an equidistant triplet. In this work we do not specify the structure 'meta'-attractors of a general multi-harmonic model. In addition, we remark that the transition between the 'meta'-attractors follows a non-trivial dynamics which is also outside of the scope of this article.

\bibliographystyle{alpha}
\bibliography{biblio}
\end{document}